\documentclass[12pt]{amsart}
\usepackage{graphicx}
\usepackage{amsfonts}
\usepackage{amssymb}
\usepackage{mathrsfs}
\usepackage[cp1250]{inputenc}
\usepackage{cite}
\newtheorem{theorem}{Theorem}[section]
\newtheorem{proposition}[theorem]{Proposition}
\newtheorem{corollary}[theorem]{Corollary}
\newtheorem{lemma}[theorem]{Lemma}
\theoremstyle{definition}

\newtheorem{example}[theorem]{Example}

\numberwithin{equation}{section}

\begin{document}
\title{Compressions of $k^{th}$--order slant Toeplitz operators to model spaces}
\author{Bartosz {\L}anucha, Ma{\l}gorzata Michalska}

\address{
Bartosz {\L}anucha,  \newline Institute of Mathematics,
\newline Maria Curie-Sk{\l}odowska University, \newline pl. M.
Curie-Sk{\l}odowskiej 1, \newline 20-031 Lublin, Poland}
\email{bartosz.lanucha@poczta.umcs.lublin.pl}

\address{
Ma{\l}gorzata Michalska,  \newline Institute of Mathematics,
\newline Maria Curie-Sk{\l}odowska University, \newline pl. M.
Curie-Sk{\l}odowskiej 1, \newline 20-031 Lublin, Poland}
\email{malgorzata.michalska@poczta.umcs.lublin.pl}

%\date{\today}
\begin{abstract}
In this paper we consider compressions of $k^{th}$--order slant Toeplitz operators to the backward shift invariant subspaces of the classical Hardy space $H^2$. In particular, we characterize these operators using compressed shifts and finite rank operators of special kind.
\end{abstract}
\maketitle

\baselineskip1.4\baselineskip

\section{Introduction}

Let $\mathbb{T}=\{z:|z|=1\}$ and denote by $L^2=L^2(\mathbb{T})$ the space of all Lebesgue measurable functions on $\mathbb{T}$ with square integrable modulus, and by $L^{\infty}=L^{\infty}(\mathbb{T})$ -- the space of all Lebesgue measurable and essentially bounded functions. Moreover, let $H^2$ be the classical Hardy space in the disk $\mathbb{D}=\{z:|z|<1\}$. As usual, we will view $H^2$ as a space of functions analytic in $\mathbb{D}$ or (via radial limits) as a closed subspace of $L^2$ (see \cite{duren, bros} for details). The orthogonal projection from $L^2$ onto $H^2$ will be denoted by $P$.

For $\varphi\in L^{\infty}$ the classical Toeplitz operator $T_{\varphi}$ is defined on $H^2$ by
  $$T_{\varphi}f=PM_{\varphi} f,\quad f\in H^2,$$
where $M_{\varphi}:L^2\to L^2$ is the multiplication operator $f\mapsto \varphi f$. For $\varphi \in L^2$ the operators $T_{\varphi}$ and $M_{\varphi}$ can be defined on a dense subset of $H^2$ and $L^2$, respectively (namely, on bounded functions).

Classical Toeplitz operators, as compressions of multiplication operators to $H^2$, are among the most studied linear operators on the Hardy space. Their study goes back to the beginning of the 20th century (for details and references see, e.g., \cite{ruben}). In recent years, compressions of multiplication operators to model spaces were widely studied. Model spaces are the non--trivial subspaces of $H^2$ which are invariant for the backward shift operator $S^{*}=T_{\bar z}$. Each of these subspaces is of the form $K_{\alpha}=H^2\ominus \alpha H^2$, where $\alpha$ is an inner function: a function analytic and bounded in $\mathbb{D}$ ($\alpha\in H^{\infty}$) such that its boundary values are of modulus one a.e. on $\mathbb{T}$.

For an inner function $\alpha$ and for $\varphi\in L^2$ let
$$A_{\varphi}^{\alpha}f=P_{\alpha}({\varphi}f),\quad f\in K_{\alpha}^{\infty}=K_{\alpha}\cap H^{\infty},$$
where $P_{\alpha}$ is the orthogonal projection from $L^2$ onto $K_{\alpha}$. Note that $A_{\varphi}^{\alpha}$ is densely defined since $K_{\alpha}^{\infty}$ is a dense subset of $K_{\alpha}$. These operators, called truncated Toeplitz operators, gained attention in 2007 with D. Sarason's paper \cite{s}. Ever since, truncated Toeplitz operators have been under constant intensive study, which led to many interesting results and applications (see \cite{gar3} and references therein). More recently, authors in \cite{part,part2} and \cite{ptak} introduced the so--called asymmetric truncated Toeplitz operators: for two inner functions $\alpha$, $\beta$ and for $\varphi\in L^2$ an asymmetric truncated Toeplitz operator $A_{\varphi}^{\alpha,\beta}$ is defined by
$$A_{\varphi}^{\alpha,\beta}f=P_{\beta}({\varphi}f),\quad f\in K_{\alpha}^{\infty}.$$
As above, $A_{\varphi}^{\alpha,\beta}$ is densely defined. Moreover, $A_{\varphi}^{\alpha,\alpha}=A_{\varphi}^{\alpha}$. These operators were then studied in \cite{BM, blicharz1,blicharz2,BL,BL3}.

Recall that if $\displaystyle{\varphi(z)=\sum_{j=-\infty}^{\infty}a_jz^j\in L^{\infty}}$, then for each $n\in\mathbb{Z}$,  $\displaystyle{M_{\varphi}(z^n)=\sum_{j=-\infty}^{\infty}a_{j-n}z^j}$ and the matrix of $M_{\varphi}$ with respect to the standard basis $\{z^n:\ n\in\mathbb{Z}\}$ is given by
\begin{displaymath}
\left[\begin{array}{cccccc}
\ddots&\vdots&\vdots&\vdots&\vdots&\\
\ldots&a_{0}&a_{-1}&a_{-2}&a_{-3}&\ldots\\
\ldots&a_1&a_{0}&a_{-1}&a_{-2}&\ldots\\
\ldots&a_2&a_{1}&a_{0}&a_{-1}&\ldots\\
\ldots&a_3&a_2&a_{1}&a_{0}&\ldots\\
&\vdots&\vdots&\vdots&\vdots&\ddots
\end{array}\right].
\end{displaymath}
Similarly, for each $n\in\mathbb{N}_0$,  $\displaystyle{T_{\varphi}(z^n)=\sum_{j=0}^{\infty}a_{j-n}z^j}$ and $T_{\varphi}$ is the operator on $H^2$ represented by the matrix (with respect to $\{z^n:\ n\in\mathbb{N}_0\}$)
\begin{displaymath}
\left[\begin{array}{ccccc}
a_{0}&a_{-1}&a_{-2}&a_{-3}&\ldots\\
a_1&a_{0}&a_{-1}&a_{-2}&\ldots\\
a_2&a_{1}&a_{0}&a_{-1}&\ldots\\
a_3&a_2&a_{1}&a_{0}&\ldots\\
\vdots&\vdots&\vdots&\vdots&\ddots\\
\end{array}\right].
\end{displaymath}

In \cite{AB1, AB}, for $k\in\mathbb{N}$, $k\geq 2$, and $\displaystyle{\varphi(z)=\sum_{j=-\infty}^{\infty}a_jz^j\in L^{\infty}}$, the authors define the $k^{th}$--order slant Toeplitz operator $U_{\varphi}$ on $L^2$ by $\displaystyle{U_{\varphi}(z^n)=\sum_{j=-\infty}^{\infty}a_{kj-n}z^j}$ for $n\in\mathbb{Z}$, that is, as the operator on $L^2$ represented by the matrix
\begin{displaymath}
\left[\begin{array}{cccccc}
\ddots&\vdots&\vdots&\vdots&\vdots&\\
\ldots&a_{0}&a_{-1}&a_{-2}&a_{-3}&\ldots\\
\ldots&a_{k}&a_{k-1}&a_{k-2}&a_{k-3}&\ldots\\
\ldots&a_{2k}&a_{2k-1}&a_{2k-2}&a_{2k-3}&\ldots\\
\ldots&a_{3k}&a_{3k-1}&a_{3k-2}&a_{3k-3}&\ldots\\
&\vdots&\vdots&\vdots&\vdots&\ddots
\end{array}\right].
\end{displaymath}
Equivalently, $U_{\varphi}:L^2\to L^2$, $\varphi\in L^{\infty}$, is given by
  $$U_{\varphi}f=W_kM_{\varphi}f,\quad f\in L^2,$$
where
  $$W_k(z^n)=\begin{cases}
               z^m& \text{if }\frac{n}{k}=m\in\mathbb{Z},\\
               0&\text{if } \frac{n}{k}\not\in\mathbb{Z}.
             \end{cases}$$
The authors also consider $V_{\varphi}$, the compression of $U_{\varphi}$ to $H^2$, defined by
  $$V_{\varphi}f=PU_{\varphi}f,\quad f\in H^2,$$
  and represented by the matrix
\begin{displaymath}
\left[\begin{array}{ccccc}
a_{0}&a_{-1}&a_{-2}&a_{-3}&\ldots\\
a_{k}&a_{k-1}&a_{k-2}&a_{k-3}&\ldots\\
a_{2k}&a_{2k-1}&a_{2k-2}&a_{2k-3}&\ldots\\
a_{3k}&a_{3k-1}&a_{3k-2}&a_{3k-3}&\ldots\\
\vdots&\vdots&\vdots&\vdots&\ddots\\
\end{array}\right].
\end{displaymath}
  Note that if $\varphi\in L^2$, then $U_\varphi$ and $V_\varphi$ are densely defined. For $k=2$, operators $U_{\varphi}$ (called slant Toeplitz operators) and their compressions to $H^2$ were first studied in \cite{Ho, ZA} (see also \cite{Ho1,Ho2,Ho3}). These operators have connections with wavelet theory and dynamical systems (see, e.g., \cite{GMW,Ho1,V}). In \cite{LL,LL2} the authors investigate commutativity of $k^{th}$--order slant Toeplitz operators. Observe that if we consider $k=1$ in the above definitions, then $U_{\varphi}=M_{\varphi}$ and $V_{\varphi}=T_{\varphi}$.

Here we study compressions of $k^{th}$--order slant Toeplitz operators to model spaces. For a fixed $k\in\mathbb{N}$, two inner functions $\alpha$, $\beta$ and for $\varphi\in L^2$ define
  $$U_{\varphi}^{\alpha,\beta}f=P_{\beta}U_{\varphi}f=P_{\beta}W_k(\varphi f),\quad f\in K_{\alpha}^{\infty},$$
and denote by $\mathcal{S}_k(\alpha,\beta)$ the set of all the compressions $U_\varphi^{\alpha,\beta}$, $\varphi\in L^2$, which can be boundedly extended to $K_\alpha$.

In Section 2 we investigate conditions on $\varphi$ which imply that $U_{\varphi}^{\alpha,\beta}=0$. We show that, in contrast with the case of $U_{\varphi}$ and $V_{\varphi}$ \cite{AB1,AB}, $U_{\varphi}^{\alpha,\beta}$ can be equal to the zero operator for $\varphi$ not equal to zero.

%consider boundedness of $U_{\varphi}^{\alpha,\beta}$ and the uniqueness of its symbol. It is known that $M_{\varphi}$ and $T_{\varphi}$  are uniquely determined by their symbols and can be extended to bounded operators on the whole space if and only if $\varphi\in L^{\infty}$. The same is true (?) for $U_{\varphi}$ and $V_{\varphi}$ \cite{ref}. We show that the symbol of $U_{\varphi}^{\alpha,\beta}$ is not unique and boundedness of $U_{\varphi}^{\alpha,\beta}$ does not imply boundedness of $\varphi$.

In Section 3 we characterize operators from  $\mathcal{S}_k(\alpha,\beta)$ using the compressed shifts $S_{\alpha}=A_z^{\alpha}$ and $S_{\beta}=A_z^{\beta}$. It is well known that a bounded linear operator $T:H^2\to H^2$ is a Toeplitz operator if and only if $T-S^* TS=0$, where $S$ is the shift operator $S=T_z$. A similar characterization was given in \cite{AB1,AB} for $V_{\varphi}$. Namely, $T$ is a compression of a $k^{th}$--order slant Toeplitz operator to $H^2$ if and only if $T-S^* TS^k=0$. We show, for example, that a bounded linear operator $U$ from $K_{\alpha}$ into $K_{\beta}$ belongs to $\mathcal{S}_k(\alpha,\beta)$ if and only if $U-S_{\beta}U(S_{\alpha}^*)^k$ is a special kind of operator of rank at most $k+1$. This is done in the spirit of D. Sarason's characterization of truncated Toeplitz operators given in \cite{s}, where, among other results, he shows that $U$ is a truncated Toeplitz operator if and only if $U-S_{\alpha}US_{\alpha}^*$ is of rank two and special kind (see also \cite{ptak, BL, BM} for the asymmetric case).

\section{Operators from $\mathcal{S}_k(\alpha,\beta)$ equal to the zero operator}

In this section we investigate for which $\varphi\in L^2$, $U_{\varphi}^{\alpha,\beta}=0$.

We start with some basic properties of the operator $W_k$ and its adjoint $W_k^*$. Some of these properties can be found for example in \cite{Ho} (for $k=2$) and in \cite{AB1,AB}.% For the readers convenience we include short proofs.
\begin{lemma} Let $k\in\mathbb{N}$. Then
\label{lem_propert_Wk}
\begin{itemize}
  \item[(a)] $W_k^*f(z)=f(z^k)$, $|z|=1$, $f\in L^2$,
  \item[(b)] $W_k^*(f\cdot g)=W_k^*f\cdot W_k^*g$ for all $f,g\in L^2$ such that $f\cdot g\in L^2$,
  \item[(c)] $W_kW_k^*=I_{L^2}$ and $W_k^*W_k$ is the orthogonal projection from $L^2$ onto the closed linear span of $\{z^{km}:m\in\mathbb{Z}\}$,
  \item[(d)] $W_k\overline{f}=\overline{W_k f}$ and $W_k^*\overline{f}=\overline{W_k^*f}$, $f\in L^2$,
  \item[(e)] $P$ reduces $W_k$, that is, $PW_k=W_kP$ and $PW_k^*=W_k^*P$,
  \item[(f)] $M_{\varphi}W_k=W_kM_{W_k^*\varphi}$ for every $\varphi\in L^{\infty}$,
  \item[(g)] for every inner function $\alpha$ the function $W_k^*\alpha$ is also inner and $P_{\alpha}W_k=W_kP_{W_k^*\alpha}$.
\end{itemize}
\end{lemma}
\begin{proof}
The proofs of (a)--(e) are straightforward. We only prove (f) and (g).

%To prove (a) take $f,g\in L^2$ with $f(z)=\sum\limits_{n=-\infty}^{\infty} a_n z^n$ and $g(z)=\sum\limits_{n=-\infty}^{\infty} b_n z^n$. Then
 % \begin{align*}
  %  \left\langle W_k g, f\right\rangle &=\left\langle \sum_{m=-\infty}^{\infty} b_{mk} z^{m}, \sum_{n=-\infty}^{\infty} a_n z^n\right\rangle =\sum_{m=-\infty}^{\infty} b_{mk} \overline{a_m} \\
   % &=\left\langle \sum_{n=-\infty}^{\infty} b_{n} z^{n}, \sum_{m=-\infty}^{\infty} a_m z^{mk}\right\rangle=\left\langle g, \sum_{m=-\infty}^{\infty} a_m (z^{k})^{m}\right\rangle,
  %\end{align*}
%and so
 % $$W_k^* f(z)=W_k^*\left(\sum_{n=-\infty}^{\infty} a_n z^n\right)=\sum_{n=-\infty}^{\infty} a_n (z^k)^n.$$

%Now (b) and (c) follow from (a). It is also easy to verify that (d) holds.

%For $f(z)=\sum\limits_{n=-\infty}^{\infty} a_n z^n \in L^2$ we have
 % $$PW_k f=PW_k\left(\sum\limits_{n=-\infty}^{\infty} a_n z^n\right)=P\left(\sum\limits_{m=-\infty}^{\infty} a_{mk} z^m\right)=\sum\limits_{m=0}^{\infty} a_{mk} z^m$$
%and
 % $$W_kP f=W_kP\left(\sum\limits_{n=-\infty}^{\infty} a_n z^n\right)=W_k\left(\sum\limits_{n=0}^{\infty} a_n z^n\right)=\sum\limits_{m=0}^{\infty} a_{mk} z^m,$$
%which means that (e) holds.

To show (f) fix $\varphi\in L^{\infty}$ and take any $f,g\in L^2$. Then by (b) and (d),
\begin{align*}
  \left\langle M_{\varphi}W_kf,g\right\rangle&= \left\langle f,W_k^*(\overline{\varphi}g)\right\rangle= \left\langle f,\overline{W_k^*\varphi}\cdot W_k^*g\right\rangle= \left\langle W_kM_{W_k^*\varphi}f, g\right\rangle.
\end{align*}

If $\alpha$ is an inner function, then $W_k^*\alpha$ is also an inner function by (a). Since $P_{\alpha}=I_{L^2}-M_{\alpha}PM_{\overline{\alpha}}$, using (e) and (f) we obtain
\begin{align*}
  P_{\alpha}W_k&=W_k-M_{\alpha}PW_kM_{\overline{W_k^*\alpha}}=W_k-M_{\alpha}W_kPM_{\overline{W_k^*\alpha}}
  \\&=W_k(I_{L^2}-M_{W_k^*\alpha}PM_{\overline{W_k^*\alpha}})=W_kP_{W_k^*\alpha}.
\end{align*}
That is, (g) holds.
\end{proof}

%Observe that by (a) $W_k^*\alpha$ is inner, whenever $\alpha$ is inner.
	For the reminder of this section fix $k\in\mathbb{N}$.
	
\begin{proposition}
\label{prop_zero_U}
Let $\alpha$ and $\beta$ be two inner functions and let $\varphi\in L^2$. If $\varphi\in \overline{\alpha H^2}+(W_k^*\beta) H^2$, then $U_{\varphi}^{\alpha,\beta}=0$.
\end{proposition}
\begin{proof}
Assume that $\varphi=\overline{\alpha h_1}+W_k^*\beta\cdot h_2$ for some $h_1,h_2\in H^2$. Then for all $f\in K_\alpha^\infty$ and $g\in K_\beta^\infty$ we have (by Lemma \ref{lem_propert_Wk}(g)),
\begin{align*}
  \left\langle U^{\alpha,\beta}_\varphi f,g\right\rangle&=\left\langle \varphi f,W^*_k g\right\rangle
  =\left\langle \overline{\alpha h_1} f,W^*_k g\right\rangle +\left\langle W_k^*\beta\cdot h_2 f,W^*_k g\right\rangle\\
  &=\left\langle f,\alpha h_1 \cdot W^*_k g\right\rangle +\left\langle P_\beta W_k(W_k^*\beta\cdot h_2 f),g\right\rangle
  =\left\langle W_kP_{W^*_k\beta}(W_k^*\beta\cdot h_2 f),g\right\rangle=0.
\end{align*}
\end{proof}

\begin{corollary}
If $U=U^{\alpha,\beta}_\varphi$, then $\varphi$ can be chosen from $\overline{K_{\alpha}}+K_{W_k^*\beta}$.
\end{corollary}
		
Hence $U^{\alpha,\beta}_\varphi$ is not uniquely determined by its symbol. Moreover, boundedness of the symbol is not necessary for the boundedness of $U_\varphi^{\alpha,\beta}$.

In general, the implication in Proposition \ref{prop_zero_U} cannot be reversed as the following example shows.

\begin{example}\label{ex1}
Let $k=2$, $\alpha(z)=z^4$ and $\beta(z)=z^3$. In that case, for $\varphi=\sum_{n=-\infty}^{\infty}a_nz^n\in L^2$, the operator $U_{\varphi}^{\alpha,\beta}$ is represented by the matrix
$$\left[\begin{array}{cccc}a_0&a_{-1}&a_{-2}&a_{-3}\\
a_{2}&a_{1}&a_{0}&a_{-1}\\
a_{4}&a_{3}&a_{2}&a_{1}\end{array}\right].$$
Hence, if for example $\varphi(z)=z^5$, then $U_{\varphi}^{\alpha,\beta}=0$. But here $$z^5\notin \overline{\alpha H^2}+(W_k^*\beta) H^2=\overline{z^4 H^2}+z^6 H^2.$$% and so $z^5\notin \overline{\alpha H^2}+(W_k^*\beta) H^2$.
\end{example}		

Observe that $U^{\alpha,\beta}_\varphi$ can be seen as a composition of $W_k$ and an asymmetric truncated Toeplitz operator. Indeed, using Lemma \ref{lem_propert_Wk}(g), for $f\in K_{\alpha}^{\infty}$ we get
$$U_{\varphi}^{\alpha,\beta}f=P_{\beta}W_k({\varphi}f)=W_kP_{W_k^*\beta}({\varphi}f)=W_kA_{\varphi}^{\alpha,W_k^*\beta}f.$$ Now Proposition \ref{prop_zero_U} follows from the fact that $A_{\varphi}^{\alpha,\beta}=0$ if and only if $\varphi\in \overline{\alpha H^2}+\beta H^2$ (see \cite{ptak, blicharz1}).

Let $\alpha$ be an inner function. Then the operator $C_{\alpha}:L^2\to L^2$, defined by the formula
\begin{align*}
%\label{def_Conj}
C_{\alpha}f(z)=\alpha(z)\overline{z}\overline{f(z)},\quad |z|=1,
\end{align*}
is a conjugation on $L^2$ (an antilinear, isometric involution). Moreover, $C_{\alpha}(K_{\alpha})=K_{\alpha}$ (see \cite{gp} and \cite[Chapter 8]{bros}). % which preserves $K_{\alpha}$ (see \cite[Subection 2.3]{s}). We will denote $\widetilde{f}=C_\alpha f$ and in particular the conjugate kernel function $\widetilde{k}_w^\alpha=C_\alpha k_w^\alpha$.
It is known that all truncated Toeplitz operators are  $C_{\alpha}$--symmetric, that is, $ C_{\alpha}A_{\varphi}^{\alpha}C_{\alpha}=(A_{\varphi}^{\alpha})^*=A_{\overline{\varphi}}^{\alpha}$ \cite{s}. It was observed in \cite{BL} that $ C_{\beta}A_{\varphi}^{\alpha,\beta}C_{\alpha}=A_{\overline{\alpha\varphi}\beta}^{\alpha,\beta}$ (see also \cite{BM}). Here we have
\begin{proposition}
\label{lem_cong_U}
Let $\alpha$, $\beta$ be two inner functions and let $\varphi\in L^2$. Then
  $$C_\beta U_\varphi^{\alpha,\beta} C_\alpha=U_\psi^{\alpha,\beta}$$
with $\psi=\overline{z^{k-1}\alpha\varphi} W_k^* \beta$.
\end{proposition}
\begin{proof}
Let $\varphi\in L^2$. Then for $f\in K_\alpha^\infty$, $g\in K_\beta^\infty$ we have, by Lemma \ref{lem_propert_Wk},
\begin{align*}
  \left\langle C_\beta U_\varphi^{\alpha,\beta} C_\alpha f,g \right\rangle&=\left\langle C_\beta P_\beta W_k M_\varphi C_\alpha f,g \right\rangle
  =\left\langle C_\beta g, P_\beta W_k (\varphi C_\alpha f) \right\rangle=\left\langle C_\beta g, W_k (\varphi C_\alpha f) \right\rangle\\
  &=\left\langle \overline{z}\beta \overline{g}, W_k (\varphi \overline{z}\alpha \overline{f}) \right\rangle
  =\left\langle W_k^*(\overline{z}\beta \overline{g}), \varphi \overline{z}\alpha \overline{f} \right\rangle
%  =\left\langle \overline{W_k^*g}, z^{k-1}\alpha\varphi \overline{W_k^*\beta}\cdot\overline{f} \right\rangle\\
  =\left\langle \overline{z^{k-1}\alpha\varphi} W_k^*\beta \cdot f,W_k^* g \right\rangle\\
  &=\left\langle P_\beta W_k M_{\overline{z^{k-1}\alpha\varphi} W_k^*\beta}f,g \right\rangle
  =\left\langle U_{\overline{z^{k-1}\alpha\varphi} W_k^*\beta}^{\alpha,\beta}f,g \right\rangle.
\end{align*}
\end{proof}
	
As a corollary to Lemma \ref{lem_cong_U} we get the following.

\begin{corollary}
\label{cor_cong_slant_k}
Let $\alpha$, $\beta$ be two inner functions and let $U$ be a bounded linear operator from $K_\alpha$ into $K_\beta$. Then $U\in\mathcal{S}_k(\alpha,\beta)$ if and only if $C_\beta UC_\alpha\in\mathcal{S}_k(\alpha,\beta)$.
\end{corollary}

\begin{proposition}
\label{prop_zero_U_bis}
Let $\alpha$, $\beta$ be two inner functions and let $\varphi\in L^2$. If $\varphi\in \overline{\alpha H^2}+ \overline{z}^{k-1}(W_k^*\beta )H^2$, then $U_{\varphi}^{\alpha,\beta}=0$.
\end{proposition}
\begin{proof}
Let $\varphi=\overline{\alpha h_1}+\overline{z}^{k-1}W_k^*\beta\cdot h_2 $ with $h_1,h_2\in H^2$. By Proposition \ref{prop_zero_U} we have
$U^{\alpha,\beta}_{\overline{\alpha h_1}}=0$ and, consequently,
  $$U_\varphi^{\alpha,\beta}=U^{\alpha,\beta}_{\overline{z}^{k-1} W_k^*\beta\cdot h_2}.$$
Now, using Lemma \ref{lem_cong_U}, we obtain
  $$C_\beta U_\varphi^{\alpha,\beta} C_\alpha=C_\beta U_{\overline{z}^{k-1} W_k^*\beta\cdot h_2}^{\alpha,\beta} C_\alpha=U_{\overline{z^{k-1}\alpha\overline{z}^{k-1} W_k^*\beta\cdot h_2}\cdot W_k^*\beta}^{\alpha,\beta}=U^{\alpha,\beta}_{\overline{\alpha h_2}}=0$$
	and it follows that $U_\varphi^{\alpha,\beta}=0$.
\end{proof}

\begin{corollary}
If $U=U^{\alpha,\beta}_\varphi$, then $\varphi$ can be chosen from $\overline{K_{\alpha}}+\overline{z}^{k-1}K_{W_k^*\beta}$.
%Every compression to $K_{\alpha}$ of a $k^{th}$--order slant Toeplitz operator has a symbol from $\overline{K_{\alpha}}+K_{W_k^*\alpha}$.
\end{corollary}

Let us note that if an inner function $\beta$ is such that $\beta(0)=0$, then $z^k$ divides $W_k^*\beta=\beta(z^k)$ and
  %$$\overline{z}^{k-1}W_k^* \alpha(z)=\sum_{j=n_0}^\infty a_j z^{(j-1)k+1}\in H^2,$$
 $\overline{z}^{k-1}W_k^* \beta$ is also an inner function. Moreover,
$$\overline{\alpha H^2}+(W_k^*\beta) H^2= \overline{\alpha H^2}+\overline{z}^{k-1}(W_k^*\beta){z}^{k-1} H^2\subset\overline{\alpha H^2}+\overline{z}^{k-1}(W_k^*\beta) H^2.$$
%  $$W_k^*\alpha\cdot H^2=\alpha(z^k) H^2\subset \overline{z}^{k-1} \alpha(z^k) H^2\subset H^2.$$

\section{Characterizations of operators from $\mathcal{S}_k(\alpha,\beta)$}

Let $\alpha$ and $\beta$ be two inner functions. In this section we will characterize operators from $\mathcal{S}_k(\alpha,\beta)$, $k\in\mathbb{N}$, using compressed shifts and finite rank operators of special kind. Recall that the compressed shift $S_{\alpha}$ is defined as $S_{\alpha}=A_z^{\alpha}=P_{\alpha}S_{|K_{\alpha}}$. As $K_{\alpha}$ is $S^*$--invariant, we have $S_{\alpha}^*=A_{\overline{z}}^{\alpha}=S^*_{|K_{\alpha}}$.

For each $n\in \mathbb{N}_0$ and $w\in\mathbb{D}$ the functional $f\mapsto f^{(n)}(w)$ is bounded on $H^2$ and so there exists $k_{w,n}\in H^2$ such that $f^{(n)}(w)=\langle f,k_{w,n}\rangle$. It is not difficult to verify that $k_{w,n}(z)=\frac{n!z^n}{(1-\overline{w}z)^{n+1}}$ and, in particular, $k_{0,n}(z)=n!z^n$. Denote $k_{w}=k_{w,0}$.
%Recall that for each $n\in \mathbb{N}_0$ and $w\in\mathbb{D}$ the functional $f\mapsto f^{(n)}(w)$ is bounded on $K_{\alpha}$ and so there exists $k_{w,n}^{\alpha,\beta}\in K_{\alpha}$ such that $f^{(n)}(w)=\langle f,k_{w,n}^{\alpha,\beta}\rangle$.

\begin{lemma}
\label{lem_aux_3_1}
Let $k\in \mathbb{N}$. For $f\in H^2$ we have
\begin{itemize}
  \item[(a)] $\displaystyle{(S^*)^kf(z)=\overline{z}^kf(z)-\sum_{j=0}^{k-1}\frac{\langle f,k_{0,j}\rangle}{j!}\overline{z}^{k-j}}$, $|z|=1$,
	%\item[(a)] $\displaystyle{(S^*)^kf(z)=\overline{z}^kf(z)-\sum_{j=0}^{k-1}\frac{f^{(j)}(0)}{j!}\overline{z}^{k-j}}$, $|z|=1$,
  \item[(b)] $W_k^*f-z^kW_k^*S^*f=\langle f,k_{0}\rangle=f(0)$.
\end{itemize}
\end{lemma}
\begin{proof}
Let $k\in \mathbb{N}$ and $f\in H^2$. Then $(S^*)^kf=(T_{\overline z})^kf=P({\overline z}^kf)$. Since $f(z)=\sum_{j=0}^{\infty}\frac{\langle f,k_{0,j}\rangle}{j!}z^j$, we have, for $|z|=1$,
$$
  (S^*)^{k}f(z)=P\left( \sum_{j=0}^{\infty}\frac{\langle f,k_{0,j}\rangle}{j!}z^{j-k}\right)=\sum_{j=k}^{\infty}\frac{\langle f,k_{0,j}\rangle}{j!}z^{j-k}=\overline{z}^kf(z)-\sum_{j=0}^{k-1}\frac{\langle f,k_{0,j}\rangle}{j!}\overline{z}^{k-j}.$$

Now, observe that
  $$W_k^*f(z)-z^kW_k^*S^*f(z)=f(z^k)-z^kW_k^* (\overline{z}f(z)-\overline{z}f(0))=f(z^k)- \left(f(z^k)- f(0)\right)= f(0),$$
which proves (b).
\end{proof}

	As $K_{\alpha}$ is a closed subspace of $H^2$, $f\mapsto f^{(n)}(w)$ is also bounded on $K_{\alpha}$ for each $n\in \mathbb{N}_0$ and $w\in\mathbb{D}$. Then $f^{(n)}(w)=\langle f,k_{w,n}^{\alpha}\rangle$ for all $f\in K_{\alpha}$ and $k_{w,n}^{\alpha}=P_{\alpha}k_{w,n}$ (see, e.g., \cite[p. 204]{fm}). Denote $\widetilde{k}_{w,n}^{\alpha}=C_{\alpha}{k}_{w,n}^{\alpha}$. In particular,
	$$ k_{w}^{\alpha}(z)=P_{\alpha}k_{w}(z)=\frac{1-\overline{\alpha(w)}\alpha(z)}{1-\overline{w}z}\quad\text{and}\quad \widetilde{k}_{w}^{\alpha}(z)=\frac{\alpha(z)-\alpha(w)}{z-w}.$$
	
\begin{lemma}
\label{lem_aux_3_2}
Let $k\in \mathbb{N}$ and $\varphi\in L^2$. Then
  $$U_{\varphi}^{\alpha,\beta}-S_{\beta}U_{\varphi}^{\alpha,\beta}(S_{\alpha}^*)^k=k_0^{\beta}\otimes \chi+\sum_{j=0}^{k-1}\psi_j\otimes k_{0,j}^{\alpha},$$
where the equality holds on $K_{\alpha}^{\infty}$ with
  $$\chi=P_{\alpha}(\overline{\varphi})\quad\text{and}\quad \psi_j=\tfrac1{j!}S_{\beta}P_{\beta}W_k(\varphi \overline{z}^{k-j}),\ 0\leq j\leq k-1.$$
\end{lemma}
\begin{proof}
Let $f\in K_\alpha^\infty$ and $g\in K_\beta^\infty$. By Lemma \ref{lem_aux_3_1}(a) we have
\begin{align*}
  \left\langle \left(U_\varphi^{\alpha,\beta}-S_\beta U_\varphi^{\alpha,\beta}(S_{\alpha}^*)^k\right) f,g \right\rangle&=
  \left\langle \varphi f,W_k^*g \right\rangle -\left\langle \varphi(S_{\alpha}^*)^k f,W_k^*S_\beta^*g \right\rangle\\
  &=\left\langle \varphi f,W_k^*g \right\rangle -\left\langle \varphi(S^*)^k f,W_k^*S^*g \right\rangle\\
  &=\left\langle \varphi f,W_k^*g \right\rangle -\left\langle \varphi\overline{z}^k f,W_k^*S^*g \right\rangle
  +\sum_{j=0}^{k-1}\frac{\left\langle f,k_{0,j}\right\rangle}{j!}\left\langle \varphi\overline{z}^{k-j},W_k^*S^*g \right\rangle\\
  &=\left\langle \varphi f,W_k^*g -z^k W_k^*S^*g \right\rangle
  +\sum_{j=0}^{k-1}\frac{1}{j!}\left\langle f,k_{0,j}^\alpha\right\rangle\left\langle S_\beta P_\beta W_k(\varphi\overline{z}^{k-j}),g \right\rangle.
\end{align*}
By Lemma \ref{lem_aux_3_1}(b),
\begin{align*}
  \left\langle \varphi f,W_k^*g -z^k W_k^*S^*g \right\rangle&=\langle \varphi f,\langle g,k_0^\beta \rangle\rangle
  =\langle \langle f,P_\alpha(\overline{\varphi}) \rangle k_0^\beta,g\rangle.
\end{align*}
\end{proof}

For $k=2$ the following was partly noted in \cite{Ho}.
\begin{lemma}
\label{lem_aux_3_0}
Let $k\in\mathbb{N}$ and $m\in\mathbb{N}_0$, $|m|<k$. Then
  $$W_kM_{z^m}W_k^*=\begin{cases}
                      I_{L^2} &\text{if }m=0,\\
                      0 &\text{if }0<|m|<k.
                    \end{cases}$$
\end{lemma}
\begin{proof}
Let $f\in L^2$, $f=\sum\limits_{j=-\infty}^\infty a_j z^j$. Then
  $$W_k^* f(z)=\sum_{j=-\infty}^\infty a_j z^{kj}\quad\text{and}\quad M_{z^m}W_k^* f(z)=\sum_{j=-\infty}^\infty a_j z^{kj+m}.$$
If $0<|m|<k$, then $kj+m$ is not divisible by $k$ and so $W_kM_{z^m}W_k^* f=0$. On the other hand, if $m=0$, then $W_kW_k^*f=f$ (Lemma \ref{lem_propert_Wk}(c)).
\end{proof}

\begin{lemma}
\label{lem_aux_3_3}
Let $k\in \mathbb{N}$, $\chi\in K_\alpha$ and $\psi_0,\ldots,\psi_{k-1}\in K_\beta$ be such that
  $$\psi_0(0)=\ldots=\psi_{k-1}(0)=0.$$
Then for
%\begin{align*}
%\label{eq_aux_3_3_1}
\begin{displaymath}
  \varphi=\overline{\chi}+\sum_{j=1}^k(k-j)! \left(W_k^*S_\beta^*\psi_{k-j}\right)\cdot z^j \in L^2
  \end{displaymath}
%\end{align*}
we have
\begin{equation}
\label{eq_aux_3_3_2}
  U_\varphi^{\alpha,\beta} -S_\beta  U_\varphi^{\alpha,\beta} (S_\alpha^*)^k=k_0^\beta\otimes \chi +\sum_{j=0}^{k-1} \psi_j \otimes k_{0,j}^\alpha.
\end{equation}
\end{lemma}
\begin{proof}
Note that for each $j\in\{1,\ldots,k\}$ we have $\left(W_k^*S_\beta^*\psi_{k-j}\right)\cdot z^j=\frac{\psi_{k-j}(z^k)}{z^k}z^j\in H^2_0$ since $\psi_{k-j}(0)=0$. This and the fact that $P_{\alpha}=P_{\alpha}P$ give
  $$P_\alpha(\overline{\varphi})=P_\alpha P(\overline{\varphi})=P_\alpha \chi=\chi.$$
Moreover, since $P_{\beta}=P_{\beta}P$ and $PW_k=W_k P$, for $1\leq l\leq k$ we have
\begin{align*}
  P_\beta W_k(\varphi \overline{z}^l)&=P_\beta W_kP(\varphi \overline{z}^l)
  =P_\beta W_k \left(\sum_{j=1}^k (k-j)!\left(W_k^*S_\beta^*\psi_{k-j}\right)\cdot z^{j-l}\right)\\
   &=P_\beta \left(\sum_{j=1}^k (k-j)! W_k M_{z^{j-l}} W_k^* S_\beta^* \psi_{k-j}\right)
  %&=P_\beta \sum_{m=0}^{k-l} W_k M_{z^{m}} W_k^* (k-m-l)!S_\beta^* \psi_{k-m-l}
  =(k-l)!S_\beta^* \psi_{k-l},
\end{align*}
where the last equality follows from Lemma \ref{lem_aux_3_0} and the fact that $0<|j-l|<k-1$ for each $j\in\{1,\ldots, k\}\setminus\{l\}$. Hence, for $0\leq j\leq k-1$,
  $$\frac{1}{j!}S_\beta P_\beta W_k(\varphi \overline{z}^{k-j})=\frac{1}{j!}S_\beta (j!S_\beta^*\psi_j)=S_\beta S_\beta^*\psi_j=\psi_j \quad (\psi_j(0)=0),$$
and \eqref{eq_aux_3_3_2} follows from Lemma \ref{lem_aux_3_2}.
\end{proof}

\begin{theorem}
\label{thm_char_U_1}
Let $U$ be a bounded linear operator from $K_\alpha$ into $K_\beta$. Then $U\in\mathcal{S}_k(\alpha,\beta)$, $k\in\mathbb{N}$, if and only if there exist functions $\chi\in K_\alpha$ and $\psi_0,\ldots,\psi_{k-1}\in K_\beta$ such that
\begin{align}
\label{eq_char_TSTO}
  U-S_\beta U(S_\alpha^*)^k=k_0^\beta\otimes \chi +\sum_{j=0}^{k-1} \psi_j\otimes k_{0,j}^\alpha.
\end{align}
\end{theorem}
\begin{proof}
If $U\in\mathcal{S}_k(\alpha,\beta)$, then $U=U_\varphi^{\alpha,\beta}$ for some $\varphi\in L^2$ and it satisfies \eqref{eq_char_TSTO} by Lemma \ref{lem_aux_3_2}.

Assume that $U$ satisfies \eqref{eq_char_TSTO} for some $\chi\in K_\alpha$ and $\psi_0,\ldots,\psi_{k-1}\in K_\beta$. Without any loss of generality we can additionally assume that $\psi_0(0)=\ldots=\psi_{k-1}(0)=0$ (otherwise we would replace $\psi_j$ and $\chi$ with $\displaystyle \psi_j-\frac{\psi_j(0)}{\|k_0^\beta\|^2}k_{0}^\beta$ and $\displaystyle \chi+\sum\limits_{j=0}^{k-1} \frac{\overline{\psi_j(0)}}{\|k_0^\beta\|^2} k_{0,j}^\alpha$, respectively). By \eqref{eq_char_TSTO}, for every $l\in\mathbb{N}_0$,
  $$S_\beta^l U(S_\alpha^*)^{kl}-S_\beta^{l+1} U(S_\alpha^*)^{(l+1)k}
  =S_\beta^l k_0^\beta\otimes S_\alpha^{kl}\chi +\sum_{j=0}^{k-1} S_\beta^l\psi_j\otimes S_\alpha^{kl}k_{0,j}^\alpha.$$
It follows that for any $n\in\mathbb{N}$,
\begin{align}
\label{eq_aux_char_TSTO_1}
  U&=\sum_{l=0}^n\left(S_\beta^l k_0^\beta\otimes S_\alpha^{kl}\chi +\sum_{j=0}^{k-1} S_\beta^l\psi_j\otimes S_\alpha^{kl}k_{0,j}^\alpha\right)
  +S_\beta^{n+1} U(S_\alpha^*)^{(n+1)k}.
\end{align}
Since for every $f\in K_\alpha$,
  $$S_\beta^{n+1} U(S_\alpha^*)^{(n+1)k}f\rightarrow 0 \quad \text{as }n\to \infty,$$
we have
\begin{align}
\label{eq_aux_char_TSTO_2}
  Uf&=\sum_{l=0}^\infty \left(S_\beta^l k_0^\beta\otimes S_\alpha^{kl}\chi +\sum_{j=0}^{k-1} S_\beta^l\psi_j\otimes S_\alpha^{kl}k_{0,j}^\alpha\right)f.
\end{align}
Let now
  $$\varphi=\overline{\chi}+\psi$$
  with
  $$\psi=\sum_{j=1}^k (k-j)!(W_k^*S_\beta^* \psi_{k-j})\cdot z^j\in H^2.$$
By Lemma \ref{lem_aux_3_3},
  $$U_\varphi^{\alpha,\beta}-S_\beta U_\varphi^{\alpha,\beta} (S_\alpha^*)^k=k_0^\beta\otimes \chi+ \sum_{j=0}^{k-1}\psi_j\otimes k_{0,j}^\alpha \quad \text{on } K_\alpha^\infty,$$
and as above \eqref{eq_aux_char_TSTO_1} holds with $U_\varphi^{\alpha,\beta}$ in place of $U$ (on $K_{\alpha}^{\infty}$). For $f\in K_\alpha^\infty$ and $g\in K_\beta^\infty$ we have
\begin{align}
\label{eq_aux_char_TSTO_3}
  \left\langle S_\beta^{n+1} U_\varphi^{\alpha,\beta}(S_\alpha^*)^{(n+1)k}f,g \right\rangle =\left\langle \varphi(S_\alpha^*)^{k(n+1)} f, W_k^*(S_\beta^*)^{n+1} g \right\rangle\to 0 \quad \text{as }n\to \infty.
\end{align}
Indeed, since
\begin{align*}
  &\left\langle \varphi(S_\alpha^*)^{k(n+1)} f, W_k^*(S_\beta^*)^{n+1} g \right\rangle\\
  &=  \left\langle \overline{\chi}T_{\overline{z}^{k(n+1)}} f, W_k^*(S^*)^{n+1} g \right\rangle+\left\langle (S^*)^{k(n+1)} f, \overline{\psi}W_k^*T_{\overline{z}^{n+1}} g \right\rangle\\
  &=\left\langle P\left(\overline{\chi z^{k(n+1)}} f\right), W_k^*(S^*)^{n+1} g \right\rangle
  +\left\langle (S^*)^{k(n+1)} f, \overline{\psi}W_k^*P({\overline{z}^{n+1}} g) \right\rangle\\
  &=\left\langle P\left(\overline{\chi z^{k(n+1)}} f\right), W_k^*(S^*)^{n+1} g \right\rangle
  +\left\langle (S^*)^{k(n+1)} f, P\left(\overline{z}^{k(n+1)}\overline{\psi}W_k^* g\right) \right\rangle,
\end{align*}
\eqref{eq_aux_char_TSTO_3} follows from the fact that
  $$\left|\left\langle P\left(\overline{\chi z^{k(n+1)}} f\right), W_k^*(S^*)^{n+1} g \right\rangle\right|
  \leq \|\chi f\|\cdot\|(S^*)^{n+1}g\|\to 0 \quad \text{as }n\to \infty,$$
and
\begin{align*}
  \left|\left\langle (S^*)^{k(n+1)} f, P\left(\overline{z}^{k(n+1)}\overline{\psi}W_k^* g\right) \right\rangle\right|
  \leq \|(S^*)^{k(n+1)} f\|\cdot \|\psi\cdot W_k^* g\|\to 0 \quad \text{as }n\to \infty.
\end{align*}
This means that for $f\in K_\alpha^\infty$ and $g\in K_\beta^\infty$,
  $$\left\langle U_\varphi^{\alpha,\beta}f,g \right\rangle =\left\langle \sum_{l=0}^\infty \left(S_\beta^l k_0^\beta\otimes S_\alpha^{kl}\chi +\sum_{j=0}^{k-1} S_\beta^l\psi_j\otimes S_\alpha^{kl}k_{0,j}^\alpha\right)f, g\right\rangle,$$
and by \eqref{eq_aux_char_TSTO_2},
  $$\left\langle U_\varphi^{\alpha,\beta}f,g \right\rangle =\left\langle Uf,g \right\rangle.$$
It follows that $U=U_\varphi^{\alpha,\beta}\in\mathcal{S}_k(\alpha,\beta)$.
\end{proof}

\begin{corollary}\label{symbol}
If a bounded linear operator $U:K_{\alpha}\to K_{\beta}$ satisfies \eqref{eq_char_TSTO}, then $U=U_\varphi^{\alpha,\beta}$ with
  $$\varphi=\overline{\chi} +\sum_{j=0}^{k-1} \psi_{j}(z^k)j!\overline{z}^j.$$
\end{corollary}
\begin{proof}
Assume that $U$ satisfies \eqref{eq_char_TSTO}. If $\psi_0(0)=\ldots=\psi_{k-1}(0)=0$, then following the proof of Theorem \ref{thm_char_U_1} we get that $U=U_\varphi^{\alpha,\beta}$ with
$$\varphi=\overline{\chi}+\sum_{j=0}^{k-1} j!(W_k^*S_\beta^* \psi_{j})\cdot z^{k-j}=\overline{\chi} +\sum_{j=0}^{k-1} \overline{z}^k\psi_{j}(z^k)j!{z}^{k-j}=\overline{\chi} +\sum_{j=0}^{k-1} \psi_{j}(z^k)j!\overline{z}^j.$$
Otherwise, as explained at the beginning of the proof of Theorem \ref{thm_char_U_1}, replace $\psi_j$ and $\chi$ with $\displaystyle \psi_j-\frac{\psi_j(0)}{\|k_0^\beta\|^2}k_{0}^\beta$ and $\displaystyle \chi+\sum\limits_{j=0}^{k-1} \frac{\overline{\psi_j(0)}}{\|k_0^\beta\|^2} k_{0,j}^\alpha$, respectively. Then, by the first part of the proof, $U=U_{\varphi_1}^{\alpha,\beta}$ with
%$$\varphi_1=\overline{\chi} +\|k_0^\beta\|^{-2} \sum_{j=0}^{k-1} \psi_{j}(0)\overline{k_{0,j}^\alpha}+      \sum_{j=0}^{k-1} \psi_{j}(z^k)j!\overline{z}^j-\|k_0^\beta\|^{-2} \sum_{j=0}^{k-1} \psi_{j}(0)k_{0}^\beta(z^k)j!\overline{z}^j.$$

%$$\varphi_1=\overline{\chi} +\sum_{j=0}^{k-1} \frac{\psi_{j}(0)}{\|k_0^\beta\|^2}\overline{k_{0,j}^\alpha}+      \sum_{j=0}^{k-1} \psi_{j}(z^k)j!\overline{z}^j- \sum_{j=0}^{k-1} \frac{\psi_{j}(0)}{\|k_0^\beta\|^2}k_{0}^\beta(z^k)j!\overline{z}^j.$$

$$\varphi_1=\overline{\left(\chi +\sum_{j=0}^{k-1} \frac{\overline{\psi_{j}(0)}}{\|k_0^\beta\|^2}{k_{0,j}^\alpha}\right)}+      \sum_{j=0}^{k-1} \left(\psi_{j}(z^k)-\frac{\psi_{j}(0)}{\|k_0^\beta\|^2}k_{0}^\beta(z^k)\right)j!\overline{z}^j.$$
However, in that case
\begin{displaymath}
\begin{split}
\varphi_1-\varphi&= \sum_{j=0}^{k-1} \frac{\psi_{j}(0)}{\|k_0^\beta\|^{2}}\overline{k_{0,j}^\alpha}-  \sum_{j=0}^{k-1} \frac{\psi_{j}(0)}{\|k_0^\beta\|^{2}}(1-\overline{\beta(0)}\beta(z^k))j!\overline{z}^j\\
&= \sum_{j=0}^{k-1} \frac{\psi_{j}(0)}{\|k_0^\beta\|^{2}}\overline{P_\alpha(j!z^j)}-  \sum_{j=0}^{k-1} \frac{\psi_{j}(0)}{\|k_0^\beta\|^{2}}j!\overline{z}^j+\overline{\beta(0)} \beta(z^k)\sum_{j=0}^{k-1} \frac{\psi_{j}(0)}{\|k_0^\beta\|^{2}}j!\overline{z}^j\\
&=\overline{\left( P_{\alpha}\left(\sum_{j=0}^{k-1} \frac{\overline{\psi_{j}(0)}}{\|k_0^\beta\|^{2}}j!{z}^j\right)-\sum_{j=0}^{k-1} \frac{\overline{\psi_{j}(0)}}{\|k_0^\beta\|^{2}}j!{z}^j   \right)}+\overline{z}^{k-1} W_k^*\beta\cdot \overline{\beta(0)} \sum_{j=0}^{k-1} \frac{\psi_{j}(0)}{\|k_0^\beta\|^{2}}j!{z}^{k-1-j}\\&\in \overline{\alpha H^2}+ \overline{z}^{k-1}(W_k^*\beta) H^2
\end{split}
\end{displaymath}
and so, by Proposition \ref{prop_zero_U_bis}, $U=U_{\varphi_1}^{\alpha,\beta}=U_{\varphi}^{\alpha,\beta}$.
\end{proof}

Note that if $\psi_0(0)=\ldots=\psi_{k-1}(0)=0$, then the pairwise orthogonal functions $\psi_{j}(z^k)j!\overline{z}^j$, $0\leq j\leq k-1$, all belong to $H^2_0$.

\begin{corollary}
If $U\in\mathcal{S}_k(\alpha,\beta)$, $k\in\mathbb{N}$, then there exist functions $\chi\in K_\alpha$ and $\psi_0,\ldots,\psi_{k-1}\in K_\beta$ such that $U=U_\varphi^{\alpha,\beta}$ with
  $$\varphi=\overline{\chi} +\psi_0\left(z^k\right) +\frac{1}{z}\psi_1\left(z^k\right) +\frac{1}{z^2}\psi_2\left(z^k\right) +\ldots +\frac{1}{z^{k-1}}\psi_{k-1}\left(z^k\right)$$
and $\psi_0(0)=\ldots=\psi_{k-1}(0)=0$. Moreover, the above decomposition is orthogonal.
\end{corollary}

\begin{corollary}
\label{thm_char_U_2}
Let $U$ be a bounded linear operator from $K_\alpha$ into $K_\beta$. Then $U\in\mathcal{S}_k(\alpha,\beta)$, $k\in\mathbb{N}$, if and only if there exist functions $\chi\in K_\alpha$ and $\psi_0,\ldots,\psi_{k-1}\in K_\beta$ such that
\begin{align}
\label{eq_char_TSTO_2}
  U-S^*_\beta US_\alpha^k=\widetilde{k}_0^\beta\otimes \chi +\sum_{j=0}^{k-1} \psi_j\otimes \widetilde{k}_{0,j}^\alpha.
\end{align}
\end{corollary}
\begin{proof}
By Corollary \ref{cor_cong_slant_k}, $U\in\mathcal{S}_k(\alpha,\beta)$ if and only if $C_\beta UC_\alpha\in\mathcal{S}_k(\alpha,\beta)$. By Theorem \ref{thm_char_U_1}, the latter happens if and only if
\begin{equation}
\label{eq_aux_char_U_2}
 C_\beta UC_\alpha- S_\beta (C_\beta UC_\alpha)(S_\alpha^*)^k=k_0^\beta\otimes \mu +\sum_{j=0}^{k-1} \nu_j\otimes k_{0,j}^\alpha
\end{equation}
for some functions $\mu\in K_\alpha$ and $ \nu_0,\ldots,\nu_{k-1}\in K_\beta$. Since $C_\alpha$ is an involution, \eqref{eq_aux_char_U_2} is equivalent to
\begin{displaymath}
%\begin{split}
  C_\beta^2 UC_\alpha^2- C_\beta S_\beta C_\beta UC_\alpha(S_\alpha^*)^kC_\alpha=C_\beta(k_0^\beta\otimes \mu)C_\alpha +\sum_{j=0}^{k-1} C_\beta(\nu_j\otimes k_{0,j}^\alpha)C_\alpha.%\\
	%
%\end{split}
\end{displaymath}
By $C_\alpha$-symmetry of the compressed shift, the above can be written as
$$U-S^*_\beta US_\alpha^k=\widetilde{k}_0^\beta\otimes \chi +\sum_{j=0}^{k-1} \psi_j\otimes \widetilde{k}_{0,j}^\alpha,$$
%\begin{align*}
%   U-  S^*_\alpha US_\alpha^k=\widetilde{k}_0^\alpha\otimes \widetilde{\mu} +\sum_{j=0}^{k-1} \widetilde{\theta}_j\otimes \widetilde{k}_{0,j}^\alpha,
%\end{align*}
with $\chi=C_\alpha\mu\in K_{\alpha}$ and $\psi_j=C_\beta\nu_j\in K_{\beta}$ for $j=0,1,\ldots,k-1$.
\end{proof}

\begin{corollary}
If $U$ satisfies \eqref{eq_char_TSTO_2}, then $U=U_\varphi^{\alpha,\beta}$ with
  $$\varphi=\beta(z^k)\overline{z}^k\overline{\chi} +\overline{\alpha}\sum_{j=0}^{k-1} \psi_{j}(z^k)j!z^{j+1}.$$
\end{corollary}

\begin{proof}
Assume that $U$ satisfies \eqref{eq_char_TSTO_2}. A reasoning similar to the one given in the proof of Corollary \ref{thm_char_U_2} shows that $C_\beta UC_\alpha$ satisfies \eqref{eq_aux_char_U_2} with $\mu =C_\alpha\chi$ and $\nu_j=C_\beta \psi_j$. By Corollary \ref{symbol}, $C_\beta UC_\alpha=U_{\psi}^{\alpha,\beta}$ with
$$\psi=\overline{\mu} +\sum_{j=0}^{k-1} \nu_{j}(z^k)j!\overline{z}^j$$
and by Proposition \ref{lem_cong_U}, $U=C_\beta U_{\psi}^{\alpha,\beta}C_\alpha=U_{\varphi}^{\alpha,\beta}$ with
\begin{displaymath}
\begin{split}
\varphi&=\overline{z^{k-1}\alpha\psi}W_k^*\beta=W_k^*\beta\left(\overline{z^{k-1}\alpha}{\mu} +\overline{z^{k-1}\alpha}\sum_{j=0}^{k-1} \overline{\nu_{j}(z^k)}j!{z}^j\right)\\
&=\beta(z^k)\left(\overline{z^{k-1}\alpha}\alpha\overline{z\chi} +\sum_{j=0}^{k-1}\overline{z^{k-1}\alpha} \overline{\beta(z^k)}z^k\psi_{j}(z^k)j!{z}^j\right)\\
&=\beta(z^k)\overline{z}^{k}\overline{\chi} +\sum_{j=0}^{k-1}\overline{\alpha} \psi_{j}(z^k)j!{z}^{j+1}.
\end{split}
\end{displaymath}
\end{proof}

%\begin{corollary}
%If $U\in\mathcal{S}_k(\alpha)$, then there exist functions $\chi, \psi_0,\psi_1,\ldots,\psi_{k-1}\in K_\alpha$ such that $U=U_\varphi^\alpha$, with
%  $$!!!!!!!!!! \varphi=\overline{\chi} +\psi_0\left(z^k\right) +\frac{1}{z}\psi_1\left(z^k\right) +\frac{1}{z^2}\psi_2\left(z^k\right) +\ldots +\frac{1}{z^{k-1}}\psi_{k-1}\left(z^k\right)$$
%and
%  $$!!!!!!!!!!! \psi_0,\psi_1,\ldots,\psi_{k-1}\in K_\alpha\cap \{k_0^\alpha\}^\bot.$$
%\end{corollary}

Using Theorem \ref{thm_char_U_1} and Corollary \ref{thm_char_U_2} one can prove the following.

\begin{corollary}
Let $U$ be a bounded linear operator from $K_\alpha$ into $K_\beta$. Then $U\in\mathcal{S}_k(\alpha,\beta)$, $k\in\mathbb{N}$, if and only if there exist functions $\chi\in K_\alpha$ and $\psi_0,\ldots,\psi_{k-1}\in K_\beta$ (possibly different for different conditions) such that one (and all) of the following conditions holds:
\begin{enumerate}
  \item[(a)] $\displaystyle S_\beta^*U- U(S_\alpha^*)^k=\widetilde{k}_0^\beta\otimes \chi +\sum_{j=0}^{k-1} \psi_j\otimes k_{0,j}^\alpha$,
  \item[(b)] $\displaystyle S_\beta U- US_\alpha^k=k_0^\beta\otimes \chi +\sum_{j=0}^{k-1} \psi_j\otimes \widetilde{k}_{0,j}^\alpha$.
\end{enumerate}
\end{corollary}

\begin{proof}
Use reasoning analogous to the one presented in the proof of \cite[Cor. 2.4]{BM}.
%Let $U$ be a bounded linear operator from $K_\alpha$ into $K_\beta$.

%We will prove (a). By Theorem \ref{thm_char_U_1} $U\in\mathcal{S}_k(\alpha,\beta)$ if and only if there exist functions $\mu\in K_\alpha$ and  $\nu_0,\nu_1,\ldots,\nu_{k-1}\in K_\beta$ such that
 % $$U-S_\beta U(S_\alpha^*)^k=k_0^\beta\otimes \mu +\sum_{j=0}^{k-1} \nu_j\otimes k_{0,j}^\alpha.$$
%Taking $S_\beta^*$ on the left hand side of the above equality we get
 % $$S_\beta^*U-S_\beta^*S_\beta U(S_\alpha^*)^k=S_\beta^*(k_0^\beta\otimes \mu) +\sum_{j=0}^{k-1} S_\beta^*(\nu_j\otimes k_{0,j}^\alpha).$$
%Now, applying identities $S^*_\beta S_\beta=I_{K_\beta} -\widetilde{k}_0^\beta \otimes  \widetilde{k}_0^\beta$ and $S^*_\beta k_0^\beta=-\overline{\beta(0)}\widetilde{k}_0^\beta$ (see, e.g., \cite{s}) we have
%\begin{align*}
 % S_\beta^*U- U(S_\alpha^*)^k&
  %=-(\widetilde{k}_0^\beta \otimes  \widetilde{k}_0^\beta)U(S_\alpha^*)^k  +(S_\beta^* k_0^\beta)\otimes \mu +\sum_{j=0}^{k-1} (S_\beta^*\nu_j)\otimes k_{0,j}^\alpha\\
  %&=-\widetilde{k}_0^\beta \otimes  (S_\alpha^k U^*\widetilde{k}_0^\beta) +(-\overline{\beta(0)}\widetilde{k}_0^\beta)\otimes \mu +\sum_{j=0}^{k-1} (S_\beta^*\nu_j)\otimes k_{0,j}^\alpha\\
  %&=\widetilde{k}_0^\beta\otimes \chi +\sum_{j=0}^{k-1} \psi_j\otimes k_{0,j}^\alpha
%\end{align*}
%with $\chi=-S_\alpha^k U^*\widetilde{k}_0^\beta - \beta(0)\mu$ and $\psi_j=S_\beta^*\nu_j$ for $j=0,1,\ldots,k-1$.

%To prove (b) we use Theorem \ref{thm_char_U_2} and identities $S_\beta S^*_\beta=I_{K_\beta}-k_0^\beta\otimes k_0^\beta$ and $S_\beta\widetilde{k}_0^\beta=-\beta(0)k_0^\beta$ (see, e.g., \cite{s}).
\end{proof}

Note that if $\text{dim}K_{\beta}=1$, then by Theorem \ref{thm_char_U_1} all bounded linear operators from $K_{\alpha}$ into $K_{\beta}$ belong to $\mathcal{S}_k(\alpha,\beta)$ for each $k\in\mathbb{N}$.

\begin{corollary}
Let $\alpha$ and $\beta$ be two inner functions and assume that $\text{dim}K_{\alpha}=m<+\infty$. If $k\geq m$, then every bounded linear operator from $K_{\alpha}$ into $K_{\beta}$ belongs to $\mathcal{S}_k(\alpha,\beta)$.
\end{corollary}
\begin{proof}
Let $U$ be a bounded linear operator from $K_\alpha$ into $K_\beta$. If $\text{dim}K_{\alpha}=m<+\infty$, then $U-S_{\beta}U(S_{\alpha}^*)^k$ has rank at most $m$, which by assumption is less or equal to $k$. Hence, there exist functions $g_0,\ldots,g_{k-1}\in K_{\alpha}$ and $f_0,\ldots,f_{k-1}\in K_{\beta}$ (some of these functions possibly equal to zero), such that
$$U-S_{\beta}U(S_{\alpha}^*)^k=\sum_{j=0}^{k-1}f_j\otimes g_j.$$
Since here the kernels $k_{0,0}^{\alpha},k_{0,1}^{\alpha},\ldots,k_{0,m-1}^{\alpha}$ are linearly independent and span $K_{\alpha}$, each $g_j$ can be written as a linear combination of these kernels and so $U$ satisfies \eqref{eq_char_TSTO}.
\end{proof}

\begin{example}
As in Example \ref{ex1} consider $\alpha(z)=z^4$ and $\beta(z)=z^3$. Then, for $k=5$ and $\varphi=\sum_{n=-\infty}^{\infty}a_nz^n\in L^2$, the operator $U_{\varphi}^{\alpha,\beta}$ is represented by the matrix
$$\left[\begin{array}{cccc}a_0&a_{-1}&a_{-2}&a_{-3}\\
a_{5}&a_{4}&a_{3}&a_{2}\\
a_{10}&a_{9}&a_{8}&a_{7}\end{array}\right].$$
It follows easily that every bounded linear operator from $K_{\alpha}$ into $K_{\beta}$ belongs to $\mathcal{S}_5(\alpha,\beta)$. Note that here $U_{z}^{\alpha,\beta}=0,$ but $$z\notin \overline{\alpha H^2}+\overline{z}^{k-1}(W_k^*\beta) H^2=\overline{z^4 H^2}+z^{11} H^2.$$
\end{example}

Finally, we describe some rank--one operators from $\mathcal{S}_k(\alpha,\beta)$.

\begin{proposition}\label{P13}
Let $\alpha$, $\beta$ be two inner functions and let $k\in \mathbb{N}$. Then for each $l\in\{0,1,\ldots,k-1\}$ the rank--one operators $\widetilde{k}_0^\beta\otimes k^\alpha_{0,l}$ and ${k}_0^\beta\otimes \widetilde{k}^\alpha_{0,l}$ belong to $\mathcal{S}_k(\alpha,\beta)$.% for $l=0,1,\ldots,k-1$.
\end{proposition}
\begin{proof}
Let $l\in\{0,1,\ldots,k-1\}$. Since $S_\beta\widetilde{k}_0^\beta=-\beta(0)k_0^\beta$, we have
\begin{align*}
  \widetilde{k}_0^\beta\otimes k^\alpha_{0,l} -S_\beta (\widetilde{k}_0^\beta\otimes k^\alpha_{0,l}) (S^*_\alpha)^k
  &=\widetilde{k}_0^\beta\otimes k^\alpha_{0,l} -(S_\beta\widetilde{k}_0^\beta)\otimes (S^k_\alpha k^\alpha_{0,l})\\
  &=\widetilde{k}_0^\beta\otimes k^\alpha_{0,l} +(\beta(0)k_0^\beta)\otimes (S^k_\alpha k^\alpha_{0,l})\\
  &=k_0^\beta\otimes( \overline{\beta(0)}S^k_\alpha k^\alpha_{0,l})+\widetilde{k}_0^\beta\otimes k^\alpha_{0,l}
\end{align*}
and by Theorem \ref{thm_char_U_1}, $\widetilde{k}_0^\beta\otimes k^\alpha_{0,l}\in\mathcal{S}_k(\alpha,\beta)$ since it satisfies \eqref{eq_char_TSTO} with $\chi=\overline{\beta(0)}(S^k_\alpha k^\alpha_{0,l})$, $\psi_l=\widetilde{k}_0^\beta$ and $\psi_j=0$ for $j\not=l$. It follows that, by Proposition \ref{lem_cong_U},
$${k}_0^\beta\otimes \widetilde{k}^\alpha_{0,l}=C_{\beta}(\widetilde{k}_0^\beta\otimes {k}^\alpha_{0,l})C_{\alpha}\in\mathcal{S}_k(\alpha,\beta).$$
\end{proof}

\begin{corollary}
For each $l\in\{0,1,\ldots,k-1\}$,
\begin{itemize}
\item[(a)] $\widetilde{k}_0^\beta\otimes k^\alpha_{0,l}=U_{\varphi}^{\alpha,\beta}$ with $\varphi=W_k^*\beta\cdot l! \overline{z}^{l+k}$,
\item[(b)] ${k}_0^\beta\otimes \widetilde{k}^\alpha_{0,l}=U_{\psi}^{\alpha,\beta}$ with $\psi=\overline{\alpha}\cdot l! {z}^{l+1}$.
\end{itemize}
\end{corollary}
\begin{proof}
	By Corollary \ref{symbol} and the proof of Proposition \ref{P13}, $\widetilde{k}_0^\beta\otimes k^\alpha_{0,l}=U_{\varphi_1}^{\alpha,\beta}$ with
	$$\varphi_1=\beta(0)\overline{\left(S_{\alpha}^kk_{0,l}^{\alpha}\right)}+\widetilde{k}_0^{\beta}(z^k)l!\overline{z}^{l}.$$
	Since for each $g\in K_{\alpha}$,
	\begin{displaymath}
	\langle S_{\alpha}^kk_{0,l}^{\alpha},g\rangle=	\langle P_{\alpha}(l!z^l),(S_{\alpha}^*)^kg\rangle=\langle l!z^l,T_{\overline{z}^k}g\rangle=\langle l!z^{l+k},g\rangle,
	\end{displaymath}
	we have $S_{\alpha}^kk_{0,l}^{\alpha}=P_{\alpha}( l!z^{l+k})$ and
	$$\varphi_1=\beta(0)\overline{P_{\alpha}( l!z^{l+k}) }+({\beta}(z^k)-\beta(0))l!\overline{z}^{l+k}=W_k^*\beta\cdot l!\overline{z}^{l+k}-\beta(0)\overline{\left( l!{z}^{l+k}-P_{\alpha}( l!z^{l+k})\right)} .$$
	Since $\varphi_1-\varphi\in\overline{\alpha H^2}$ we get $\widetilde{k}_0^\beta\otimes k^\alpha_{0,l}=U_{\varphi_1}^{\alpha,\beta}=U_{\varphi}^{\alpha,\beta}$, that is, (a) holds.
	
	Now (b) follows from Proposition \ref{lem_cong_U}.	
	\end{proof}

\end{document}